\documentclass[english,12pt]{article}

\usepackage[T1]{fontenc}
\usepackage[latin9]{inputenc}
\usepackage{amssymb,amsmath,amsthm}
\usepackage{graphicx}
\usepackage{units}
\usepackage{babel}
\usepackage[numbers,sort&compress]{natbib} 
\usepackage{verbatim}

\numberwithin{equation}{section}

\newcounter{mycounter}
\newtheorem{thm}[mycounter]{Theorem}
\newtheorem{lem}[mycounter]{Lemma}
\newtheorem{prop}[mycounter]{Proposition}

\newcommand{\dint}{\,\mathrm{d}}
\newcommand{\Rn}{\mathbb{R}}
\newcommand{\Nn}{\mathbb{N}}

\newcommand{\Lp}{\mathnormal{L}}
\newcommand{\eps}{\varepsilon}

\newcommand{\op}[1]{\mathrm{#1}}

\DeclareMathOperator{\spr}{spr}

\begin{document}
\title{A Sufficient Condition for the Existence of a Principal Eigenvalue for
  Nonlocal Diffusion Equations with Applications.}
\author{Daniel B. Smith}
\date{\today} 
\maketitle

\begin{abstract}
  Considerable work has gone into studying the properties of nonlocal diffusion
  equations. The existence of a principal eigenvalue has been a significant
  portion of this work. While there are good results for the existence of a
  principal eigenvalue equations on a bounded domain, few results exist for
  unbounded domains. On bounded domains, the Krein-Rutman theorem on Banach
  spaces is a common tool for showing existence.  This article shows that
  generalized Krein-Rutman can be used on unbounded domains and that the theory
  of positive operators can serve as a powerful tool in the analysis of
  nonlocal diffusion equations.  In particular, a useful sufficient condition
  for the existence of a principal eigenvalue is given.
\end{abstract}

\section{Introduction}

Nonlocal diffusion equations come up in a number of contexts. These range from
biology \cite{Hutson2003,Bates2007,Maly2001,Weichsel2010}, to materials science
\cite{Chen2011}, and graph theory. Much of the analysis to this point has been
done on bounded domains \cite{Coville2010}. The results presented here
represent a first step in analyzing nonlocal diffusion equations on unbounded
domains. The theory of positive operators on Banach lattices is employed to
give useful conditions to show existence of a principal eigenvalue in the
absence of compact domains. Many more results are likely to come from applying
known results about positive operators to nonlocal diffusion equations. In
particular, in Section \ref{maximum}, we use the existence of a particular type
of maximum principal implies exponential convergence to zero.

The existence of a principal eigenvalue has been used to study many nonlocal
diffusion equations, both linear and nonlinear \cite{Bates2007,
  Garcia-Melian2009}. Energy methods and Fourier analysis have been used to
provide polynomial bounds on the decay of nonlocal diffusion equations
\cite{Chasseigne2006, Ignat2009}, with exponential decay shown in some special
cases \cite{Ignat2009}. However, estimates of the principal eigenvalue would be
able to definitely show exponential decay. While much work has been done on the
existence of principal eigenvalues on bounded domains, few general results
exist for unbounded domains \cite{Ignat2012, Coville2010}. It is important to
note that, even on bounded domains, existence of a principal eigenvalue is
still not guaranteed. See, for example, the counter example in section
\ref{sec:exist}. The work here uses the theory of positive operators to find a
sufficient condition for the existence of a principal eigenvalue without any
assumptions on the boundary conditions.  The method also gives a technique for
estimating the value of the eigenvalue.  Further work is needed to characterize
the multiplicity of the eigenvalue or the existence of a spectral gap.

The general equation of interest is:
\begin{equation}\label{dyna}
  \dot{u}(x,t)=\int\limits_{\Omega}J(x,y)u(y,t)\dint y-a(x)u(x,t)
\end{equation}
where $\Omega$ is some connected, possibly unbounded subset of the real line
$\Rn$ and $u\in\Lp^p(\Omega)$ for $1\leq p<\infty$.  This equation is often
interpreted as modeling some form population dispersal. Individuals propogate
from point $y$ to point $x$ at a rate $J(x, y)$ and individuals die off at a
rate $a(x)$ depending on their location. Note that the proofs here
immediately generalizes to $\Rn^n$.  $a$ is assumed to be continuous and
bounded in the sense that there exists $c$, $c^\prime$ such that:
\begin{equation}
  0<c<a(x)<c^\prime
\end{equation}
The two additional hypotheses on $J$ are that $J(x,y)\geq0$ and $\op{J}$ is
bounded with non-zero spectrum, where $\op{J}$:
\begin{equation}
  \op{J}u = \int_\Omega J(x,y)u(y)\dint y
\end{equation}
The result in Theorem \ref{condition} gives a quite general condition for the
existence of a principal eigenvalue for $\op{L}$. There is no particular reason
that $\op{J}$ has to be of integral type, but the discussion here will be
limited to that case. If $\op{J}$ is a more general positive operator, an
appropriate compact topological space will have to be found to ensure the
conclusion from Lemma \ref{K-R-eigen}. Several applications of the theorem are
given in section \ref{applications}.

Define the two related linear operators:
\begin{equation}
  \op{L}u=\op{J}u-a u
  \qquad
  \op{A}_\lambda u=\frac{\op{J}u}{\lambda+a}
\end{equation}
$\op{L}$ is simply the operator that defines the dynamics of \eqref{dyna}, and
$\op{A}_\lambda$ is a related family of operators. It is obvious that
$\op{A}_\lambda$ is positive whenever $\lambda>-\inf a(x)$. The two operators
are related in the sense that $\op{L}$ has $\lambda>-\inf a$ as an eigenvalue
if and only if $\op{A}_\lambda$ has an eigenvalue equal to one.

The theorem can be stated as such:
\begin{thm}\label{condition}
  If:
  \begin{equation}
    \lim_{\lambda\to(-\inf a)^+}\spr(\op{A}_\lambda)>1
  \end{equation}
  where $\spr(\op{A}_\lambda)$ is the spectral radius and one of the conditions
  in Lemma \ref{K-R-eigen} holds, then $\op{L}$ has a principal eigenvalue,
  $\lambda_0\in\Rn$, with positive eigenfunction such that for any element
  $\lambda\in\sigma(\op{L})$, we have the inequality $\Re\lambda\leq\lambda_0$.
\end{thm}
The condition in Theorem \ref{condition} holds for any boundary conditions and
is not dependent upon the domain being bounded.

A quick lemma that gives a few conditions necessary for the existence of an
eigenvalue:
\begin{lem}\label{K-R-eigen}
  $\op{A}_\lambda$ has a positive eigenvalue equal to its spectral radius with
  a positive eigenfunction in $\Lp^1$ whenever $\lambda>-\inf a$ if any
  of the following conditions hold:
  \begin{enumerate}
  \item Range of $\op{J}^n\subset \Lp^p$ for $1<p<\infty$ for some $n\in\Nn$.
    \label{p>1}
  \item $J(x,y)=K(x-y)$ is of convolution type with $K$ absolutely bounded.
    \label{convolution}
  \item $\exists c(x)\in\Lp^1$ such that $J(x,y)\leq c(x)$ 
    \label{boundedness}
  \end{enumerate}
\end{lem}

\begin{proof}
  $\op{A}_\lambda$ is obviously a positive operator on the appropriate lattice
  in the sense that it preserves the positive cone. Because of the boundedness
  of $a+\lambda$, we know that $\op{A}_\lambda$ has range in the same function
  space(s) as $\op{J}$. For that reason, the rest of the proof will only be
  concerned with $\op{J}$. Also, recall that we are assuming $\op{J}$ has 
  non-zero spectrum for the entirety of this article. 

  For condition \ref{p>1}, we know that all of the $\Lp^p$ spaces where
  $1<p<\infty$ are weakly compact, so we can invoke the generalized
  Krein-Rutman theorem on locally convex topological spaces immediately since
  $\op{A}_\lambda$ has range in a compact topological space. \cite{Top-Vect}

  Condition \ref{convolution} allows us to consider the case where $\op{J}$ is
  from $\Lp^1$ to $\Lp^1$. Recall that the unit ball of $ca(\Sigma)$, the space
  of countably additive measures with respect to a $\sigma$-algebra $\Sigma$,
  is compact with respect to the weak topology, and we can again use
  Krein-Rutman to get the existence of an eigenvalue. Then, for any
  eigenfunction $u$, we can write:
  \begin{equation}\label{eigen-multiply}
    \op{J}u=(\lambda+\kappa)u
  \end{equation}
  It is easy to show that $u$ must be absolutely continuous with respect to the
  Lebesgue measure because of the convolution with a bounded
  function. Therefore, $u$ must be isometric to an $\Lp^1$ function by the
  Radon-Nikodym theorem.

  Finally, condition \ref{boundedness} allows us to make the same argument as
  for condition \ref{convolution}. It is straightforward to show that
  $\op{J}:ca(\Sigma)\to\Lp^1$ and that any eigenfunction must be an $\Lp^1$
  function.
\end{proof}

The three conditions presented in Lemma \ref{K-R-eigen} are surely not
exhaustive, as will be seen in Section \ref{applications}. More generally, if
$\op{J}$ is a bounded operator from $ca(\Sigma)\to\Lp^1$, the conclusions also
hold.

We need one more lemma about $\op{A}_\lambda$. 
\begin{lem}
  $\spr(\op{A}_\lambda)$ is a continuous, monotonically decreasing function of
  $\lambda$ which converges to 0.
\end{lem}
\begin{proof}
  First, we want to show that $\spr(\op{A}_\lambda)\to0$ as
  $\lambda\to\infty$. To show that, simply observe that we have the $\Lp^1$
  operator norm inequality:
  \begin{equation}\label{op-ineq}
    \|\op{A}_\lambda\|_{op}\leq\left\|\frac{1}{\lambda+a}\right\|_\infty
    \|J\star\|_{op}=\frac{1}{\lambda+\inf a}
  \end{equation}
  The right hand side of that equation implies the norm of $\op{A}_\lambda$
  goes to zero asymptotically, which bounds the spectral radius above. Finally,
  we want to show that the spectral radius of $\op{A}_\lambda$ is a
  monotonically decreasing, continuous function of $\lambda$. From previous
  results, we know that $\spr(\op{A}_\lambda)$ is upper-semicontinuous
  \cite{Newburgh1951}, and that $\op{A}_\lambda\leq\op{A}_\mu$ implies that
  $\spr(\op{A}_\lambda)\leq\spr(\op{A}_\mu)$ \cite{Marek1970}. Since
  $\op{A}_\lambda<\op{A}_\mu$ whenever $\lambda>\mu$, we have that
  $\spr(\op{A}_\lambda)$ is a monotonically decreasing, upper-continuous
  function.  

  It remains to show that the spectral radius function is lower-continuous.
  Define the function $c(\lambda,\lambda^\prime)$ for $\lambda$,
  $\lambda^\prime>-\inf a$:
  \begin{equation}
    c(\lambda,\lambda^\prime)
    =\sup_{x\in\Omega}\frac{\lambda+a(x)}{\lambda^\prime+a(x)}
  \end{equation}
  Observe that $c$ is a continuous function and $c(\lambda,\lambda)=1$ for all
  $\lambda>-\inf a(x)$. Assume $\lambda>\lambda^\prime$. Then,
  $c(\lambda,\lambda^\prime)>1$. We also have the two inequalities:
  \begin{equation}
    \op{A}_\lambda\leq\op{A}_{\lambda^\prime}
    \qquad
    c(\lambda,\lambda^\prime)\op{A}_\lambda\geq\op{A}_{\lambda^\prime}
  \end{equation}
  We can now directly calculate:
  \begin{equation}
    \lim_{\lambda\,\downarrow\,\lambda^\prime}\spr(\op{A}_\lambda)
    \geq\lim_{\lambda\,\downarrow\,\lambda^\prime}c(\lambda,\lambda^\prime)
    \spr\op{A}_{\lambda^\prime}=\spr(\op{A}_{\lambda^\prime}).
  \end{equation}
  which completes the proof of lower-continuity.
\end{proof}
Note that the result of the above lemma holds even withtout any of the
necessary conditions for Lemma \ref{K-R-eigen}. That fact becomes important for
the result in Section \ref{maximum}. We can now move on to the rest of the
proof.

\section{Proof of the Theorem}

First, we will show that the operator $\op{L}$ is positive resolvent. Choose
$\lambda\in\Rn$ so large that $\mu\geq\lambda$ implies $\mu\in\rho(\op{L})$ the
resolvent of $\op{L}$ and that $\spr(\op{A}_\mu)<1$. Assume $f\geq0$,
$f\in\Lp^1$, is in the range of $\mu-\op{L}$:
\begin{equation}
  \mu v-\op{L}v=f
\end{equation}
for some $v\in\Lp^1$. Expanding the above equation gives:
\begin{align}\nonumber
  \mu v-\op{J}v+a v &= f & \iff \\\nonumber
  v-\frac{\op{J}v}{\mu+a} &= \frac{f}{\mu+a} & \iff \\\nonumber
  v-\op{A}_\lambda v &= \frac{f}{\mu+a} & \iff \\
  v &= \sum_{j=0}^\infty\op{A}_\lambda^j\frac{f}{\mu+a}\geq0
\end{align}
Since we have assumed that $\spr(\op{A}_\mu)<1$, the above series converges 
and $(\mu-\op{L})^{-1}\geq0$. By \cite{Nussbaum1984}, we know that there exists
$\lambda_0$ such that:
\begin{equation}
  \spr\Big((\mu-\op{L})^{-1}\Big)=\frac{1}{\mu-\lambda_0}
\end{equation}
for all $\mu>\lambda_0$. We also have the characterization:
\begin{equation}
  \lambda_0=\inf\{\lambda\in\Rn|(\lambda-\op{L})^{-1}\geq0\}
\end{equation}
Assume $\lambda_0>-\inf a$. By \cite{Top-Vect}, we know that:
\begin{equation}
  (\op{I}-\op{A}_\lambda)^{-1}\geq0 \iff \spr(\op{A}_\lambda)<1
\end{equation}
The above condition implies that $\spr(\op{A}_\lambda)<1$ for all
$\lambda>\lambda_0$ and $\spr(\op{A}_{\lambda_0})\geq1$. The continuity of
$\spr(\op{A}_\lambda)$ implies that $\spr(\op{A}_{\lambda_0})=1$. From our
application of the Krein-Rutman theorem, we know there exists $u\geq0$ such
that:
\begin{equation}
  \op{A}_{\lambda_0}u=u
\end{equation}
From our definition of $\op{A}_\lambda$, the above implies:
\begin{equation}
  \op{L}u=\lambda_0u
\end{equation}
$\lambda_0$ is therefore our principal eigenvalue and $u$ its associated
positive eigenfunction.

All that remains to be shown is that $\lambda_0>-\inf a$. Recall the hypothesis
on the spectral radius limit:
\begin{equation}
  \lim_{\lambda\to(-\inf a)^+}\spr(\op{A}_\lambda)>1
\end{equation}
That inequality implies there exists a largest $\lambda_0>-\inf a$ such that
$\spr(\op{A}_{\lambda_0})=1$, and we are finished.

Note that in the proof of the theorem, the structure of the integral equation
only comes up in Lemma \ref{K-R-eigen}. Define the more general operator
$\op{B}$ as:
\begin{equation}
  \op{B}f = \op{K}f - af
\end{equation}
where $\op{K}$ is some positive operator. The only conditions we need for
Theorem \ref{condition} to hold are that $\op{K}$ has non-zero spectrum and the
range of $\op{K}^n$ is inside a compact topological space for some finite
$n\in\Nn$.\footnote{See Theorem 6.6 in Appendix V. and the following example in
  \cite{Top-Vect} for details.} In particular, if $\op{K}^n\rightarrow
U\subset\Lp^p$ for $1<p<\infty$, Theorem \ref{condition} holds.

\section{Applying the Theorem}\label{applications}

The result in Theorem \ref{condition} can be directly applied to prove the
existence of a principal eigenvalue for a variety of problems. First, we will
show that the existence of a maximum principal implies exponential convergence
even without the existence of a principal eigenvalue. Then, we will give a
couple of abstract propositions with and without symmetry conditions on $J$.

\subsection{Maximum Principle}\label{maximum}

We define our maximum principle similar to \cite{Coville2010} as the condition:
\begin{equation}\label{max-prin-def}
  \op{L}u\leq0\implies u\geq0
\end{equation} 
Recall that we can rewrite the above equation to read:
\begin{equation}
  \op{L}u=f\leq0 \qquad\iff\qquad (\op{I}-\op{A}_0)u=g\geq0
\end{equation}
where $-a(x)g(x)=f(x)$. Our maximum principle is now equivalent to the
assertion that $(\op{I}-\op{A}_0)^{-1}$ is a positive operator. Using a result
from Appendix 2 of \cite{Top-Vect}, we know that $(\op{I}-\op{A}_0)^{-1}$ is
positive if and only if $\spr(\op{A}_0)<1$. Recalling the proof of Theorem
\ref{condition}, we know that $\spr(\op{A})<1$ if and only if $\spr(\op{L})<0$.
Our maximum principle therefore implies exponential convergence to zero for
the nonlocal diffusion equation defined by $\op{L}$.

\subsection{Existence Propositions}\label{sec:exist}

For all of the results in this section, we will assume at least one of the
conditions in Lemma \ref{K-R-eigen} holds.

\begin{prop}\label{achieves-inf}
  Assume $a(x)$ is locally Lipschitz and reaches its infimum at $x^\star$ and
  $J(x,y)=J(y,x)$.  Also, assume the range of $\op{J}$ is in $\Lp^2$ and there
  exists an open neighborhood $U\subset\Rn^2$ of $(x^\star, x^\star)$ such that
  $J(x,y)>c$ for all $(x,y)\in U$. Then, $\op{L}$ has a principal eigenvalue.
\end{prop}

\begin{proof}
  Without loss of generality, assume $x^\star=0$. From \cite{Drnovsek2000}, we
  can get a lower bound on the spectral radius by observing that:
  \begin{equation}
    \op{A}_\lambda u\geq c u 
    \quad 
    \implies 
    \quad
    \spr(\op{A}_\lambda)\geq c
  \end{equation}
  We can then explicitly borrow the test function from \cite{Hutson2003}. Fix
  $\delta$, $\eps$ such that $J(x,y)\geq\eps$ for all $|x|$, $|y|\leq\delta$.
  For every $\gamma>0$ write $u_\gamma$:
  \begin{equation}
    u_\gamma(x)=
    \begin{cases}
      \frac{1}{\gamma+a(x)-a(0)} & \text{if } |x|<\delta \\
      0                          & \text{else}
    \end{cases}
  \end{equation}
  We have the inequalities:
  \begin{align}\nonumber
    \op{J}u_\gamma(x)&=\int\limits_{-\delta}^\delta 
    \frac{J(x,y)}{\gamma+a(y)-a(0)}\dint y \\
    &\geq \eps\int\limits_{-\delta}^\delta \frac{1}{\gamma+C|y|}
    \geq \frac{\eps}{C}\ln\left[\frac{C\delta+\gamma}{\gamma}\right]
  \end{align}
  Choose $\gamma>0$ such that the last term above is greater than 
  one and $-a(0)<\lambda<\gamma-a(0)$. Using the above inequality gives:
  \begin{align}\nonumber
    \op{A}_\lambda u_\gamma(x)&=\frac{\op{J}}{\lambda+a(x)} \\\nonumber
    &\geq\frac{1}{\gamma+a(x)-a(0)}
    \frac{\eps}{C}\ln\left[\frac{C\delta+\gamma}{\gamma}\right] \\
    &\geq u_\gamma(x)
  \end{align}
  We have now shown that $\spr(\op{A}_\lambda)\geq1$ and can invoke Theorem 
  \ref{condition}.
\end{proof}

\begin{prop}\label{converge-inf}
  Assume that $a(x)$ converges to its infimum at either $\pm\infty$, and there
  exists $c$, $c^\prime$ such that $J(x,y)>\eps$ for $|y-x|<\delta$ for all
  $x\in\Omega$. Then, $\op{L}$ has a principal eigenvalue.
\end{prop}
\begin{proof}
  Without loss of generality, assume $a(x)\to\inf a$ as $x\to\infty$. We can
  construct almost the same inequality as in the proof of Proposition 
  \ref{achieves-inf}. Choose $\lambda$, $z$ such that:
  \begin{equation}
    \frac{\eps}{\lambda+a(x)}>1 
    \quad
    \text{for }|x-z|<\delta
  \end{equation}
  Choosing $u=\mathbf{1}_{(z-\delta, z+\delta)}$ combined with the inequality in
  \cite{Drnovsek2000} completes the proof.  
\end{proof}

Note that the conditions on J for Proposition \ref{converge-inf} are easily
satisfied if $J(x,y) = f(x-y)$ and $f>c>0$ on some neighborhood of zero.
The proof of that proposition could also be easily generalized 
Without enumerating all possible cases, it is obvious that other conditions on
$J$ can be used for other similar cases, e.g. $a(x) =
\frac{1}{1+|x|}+\cos(\theta)$ and that weaker conditions on $J$ can be found
for the cases listed above.

\begin{prop}
  Assume $a(x)$ is locally Lipschitz and reaches its infimum at $x^\star$ and
  $J(x,y)=J(y,x)$.  Also, assume the range of $\op{J}$ is in $\Lp^2$ and there
  exists $\delta$, $\eps>0$ and a bounded set $U$ such that:
  \begin{equation}
    \int_U J(x,y)\dint x > \eps
    \quad
    \text{for  } |y-x^\star|<\delta \text{ a.e.}
  \end{equation}
Then, $\op{L}$ has a principal eigenvalue.
\end{prop}
\begin{proof}
  Again, we will assume $x^\star=0$ for the sake of notation. Define the
  family of inner-products $\langle\cdot,\cdot\rangle_{\lambda+a}$ by:
  \begin{equation}
    \langle f, g\rangle_{\lambda+a} = 
    \int_\Omega f(x)\bar{g}(x)\Big(\lambda+a(x)\Big)\dint x
  \end{equation}
  It is easy to observe that $\op{A}_\lambda$ is self-adjoint with respect to
  the inner-product $\langle\cdot, \cdot\rangle_{\lambda+a}$ and that
  $\langle\cdot,\cdot\rangle_{\lambda+a}$ is topologically equivalent to the
  usual $\Lp^2$ inner-product for all $\lambda>-\inf a$.

  Need to show there exists $v$ and $\lambda$ such that:
  \begin{equation}
    \frac{\langle\op{A}_\lambda v,v\rangle_{\lambda+a}}
         {\langle v,v\rangle_{\lambda+a}}>1
  \end{equation}
  Use:
  \begin{equation}
    v_\lambda = \frac{\mathbf{1}_{(-\delta,\delta)}}{\lambda+a(x)}+\mathbf{1}_U
  \end{equation}
  where $U$ and $\delta$ are as in the statement of the proposition. Choose
  some $\lambda^\star>-\inf a$ and observe that:
  \begin{equation}
    \langle v_\lambda,v_\lambda\rangle_{\lambda+a}
    \leq \langle v_{\lambda^\star}, v_{\lambda^\star}\rangle_{\lambda^\star+a}=C
  \end{equation}
  for all $\lambda\leq\lambda^\star$. 

  Define $\gamma=\lambda-\inf a$. We can now construct the sequence of
  inequalities:
  \begin{align}\nonumber
    \langle \op{A}_\lambda v_\lambda, v_\lambda\rangle_{\lambda+a} &=
    \left\langle\frac{\op{J} v_\lambda}
                     {\lambda+a},v_\lambda\right\rangle_{\lambda+a} 
    \\\nonumber
    &=\Big\langle\op{J}v_\lambda,v_\lambda\Big\rangle \\\nonumber
    &\geq\int_\Omega\frac{\op{J}\mathbf{1}_U(x)}{\lambda+a(x)}\dint x \\
    &\geq\int\limits_{-\delta}^\delta\frac{\eps}{\gamma+c|x|}\dint x
    \geq \frac{\eps}{c}\ln\left[\frac{c\delta+\gamma}{\gamma}\right]
  \end{align}
  Choosing $\lambda$ sufficiently small gives $\langle\op{A}_\lambda
  v_\lambda,v_\lambda\rangle_{\lambda+a}>C$, which gives the desired result.
\end{proof}

Finally, a counter-example to show that the Lipschitz hypothesis is necessary
when $a(x)$ reaches its global minimum and is bounded away from that minimum
elsewhere. For simplicity, we will take $\Omega=S^1$. Take the family of
systems where $J=\frac{1}{2\pi}$ and $a(x)=c|x|^\alpha+c^\prime$ where
$c,\,c^\prime>0$ and $0<\alpha<1$. Obviously, $a$ is $\alpha$-H\"{o}lder
continuous. Assume we have an eigenpair $\mu$, $v(x)$. Dividing through
wherever $a+\mu\neq0$, we have:
\begin{equation}
  v  =\frac{\int v(\omega)\dint\omega}{c|\theta|^\alpha+c^\prime+\mu}
\end{equation}
That means $\int v(\omega)\dint\omega=1$ and
$v=\frac{1}{c|\theta|^\alpha+c^\prime+\mu}$. However, if we set:
\begin{equation}
  c=2\int\limits_{S^1}|\theta|^{-\alpha}\dint\theta
\end{equation}
the above equations are not solvable for any pair $c,\,\mu$ since $\int
v(\omega)\dint\omega\leq\frac{1}{2}$ for all $\mu\geq-c^\prime$. If
$\mu<c^\prime$, the integral of the potential eigenfunction is no longer
defined. That implies that there is not any eigenvalues for $\op{A}$.

\bibliographystyle{plain}
\bibliography{orient}

\end{document}